\documentclass[psamsfonts]{amsart}

\usepackage{multicol}
\usepackage{comment}
\usepackage{tabu}
\usepackage{hhline}
\usepackage{enumerate}
\usepackage{float}
\usepackage{amsmath}
\usepackage{amsfonts}
\usepackage{amsthm}
\usepackage{amssymb}
\usepackage[pdftex]{graphicx}

\renewcommand{\mod}{\textrm{~mod~}}

\DeclareSymbolFont{extraup}{U}{zavm}{m}{n}
\DeclareMathSymbol{\varheart}{\mathalpha}{extraup}{86}
\DeclareMathSymbol{\vardiamond}{\mathalpha}{extraup}{87}

\theoremstyle{plain}
\newtheorem{thm}{Theorem}[section]
\newtheorem{lem}[thm]{Lemma}

\theoremstyle{definition}
\newtheorem{defn}{Definition}

\title{Triangulating Almost-Complete Graphs}
\author{Kim Nguyen Pham, Landon Settle, Kayla Wright, Padraic Bartlett}

\begin{document}

\begin{abstract}
A triangle decomposition of a graph $G$ is a partition of the edges of $G$ into triangles.  Two necessary conditions for $G$ to admit such a decomposition are that $|E(G)|$ is a multiple of three and that the degree of any vertex in $G$ is even; we call such graphs tridivisible.  

Kirkman's work on Steiner triple systems established that for $G \simeq K_n$, $G$ admits a triangle decomposition precisely when $G$ is tridivisible.  In 1970, Nash-Williams conjectured that tridivisiblity is also sufficient for ``almost-complete'' graphs, which for this talk's purposes we interpret as any graph $G$ on $n$ vertices with $\delta(G) \geq (1 -\epsilon)n, E(G) \geq (1 - \xi)\binom{n}{2}$ for some appropriately small constants $\epsilon, \xi$.   Nash-Williams conjectured that $\epsilon = \xi =1/4$ would suffice; in 1991, Gustavsson demonstrated in his dissertation that $\epsilon  = \xi < 10^{-24}$ suffices for all $n \equiv 3, 9 \mod 18$, and in 2015 Keevash's work on the existence conjecture for combinatorial designs established that some value of $\epsilon$ existed for any $n$.

In this paper, we prove that for any $\epsilon < \frac{1}{432}$, there is a constant $\xi$ such that any $G$ with $\delta(G) \geq (1 - \epsilon)n$ and $|E(G)| \geq (1 - \xi)\binom{n}{2}$ admits such a decomposition, and offer an algorithm that explicitly constructs such a triangulation.  Moreover, we note that our algorithm runs in polynomial time on such graphs.  (This last observation contrasts with Holyer's result that finding triangle decompositions in general is a NP-complete problem.)

\end{abstract}

\maketitle

\section{Introduction}

Given any graph $G$, we say that $G$ admits a \textbf{triangle decomposition} if we can partition the edges of $G$ into disjoint copies of $K_3$.  In order for such a decomposition to exist, there are two ``obviously necessary'' conditions that $G$ needs to satisfy: the degree of any vertex must be even, as each triangle uses two edges to each of its vertices, and the total number of edges in $G$ must be a multiple of 3.  Taking after Keevash \cite{keevash2015counting}, we say that a graph $G$ is \textbf{tridivisible} if it meets these two conditions.

In general, tridivisibility is not sufficient to ensure that a graph has a triangle decomposition; consider $C_6$, for example.  However, for the complete graphs $K_n$, a triangle decomposition is equivalent to a Steiner triple system\footnote{A \textbf{Steiner triple system} $S(t,k,n)$ is any set $A$ of size $n$ along with a collection $B$ of subsets of $A$, such that the following two properties hold:
\begin{itemize}
\item Every element of $B$ has size $k$.
\item Every subset of $A$ of size $t$ is a subset of exactly one element of $B$.
\end{itemize}} $S(2,3,n)$; that is, a triangle decomposition of $K_n$ is any collection of blocks of $V(K_n)$ all of size three, such that any subset of size two of $V(K_n)$ shows up in exactly one block.  Kirkman \cite{Kirkman_1847}, \cite{lindner2008design} proved that tridivisibility is indeed a sufficient condition to ensure that a $S(2,3,n)$ exists; that is, $K_n$ admits a triangle decomposition if and only if $n$ is congruent to 1 or 3 mod 6.

In 1970, Nash-Williams \cite{Nash_Williams_1970} asked the following natural question: for what families of graphs is tridivisibility a sufficient condition to ensure that a triangle decomposition exists?  In particular, he asked if there was a constant $\epsilon > 0$ such that any tridivisible graph $G$ on $n$ vertices with $\delta(G) \geq (1-\epsilon)n$ must admit a triangle decomposition, and conjectured that $\epsilon = \frac{1}{4}$ would suffice.

If any value of $\epsilon$ does exist, then it is at most $1/4$. Following Gustavsson \cite{Gustavsson_1991}, look at the strong graph product $G = C_4 \boxtimes K_n$, i.e. the graph formed by taking four copies of $K_n$, placing them on the corners of a square, and connecting every pair of vertices from adjacent corners.
\begin{figure}[H]
\includegraphics[width=2in]{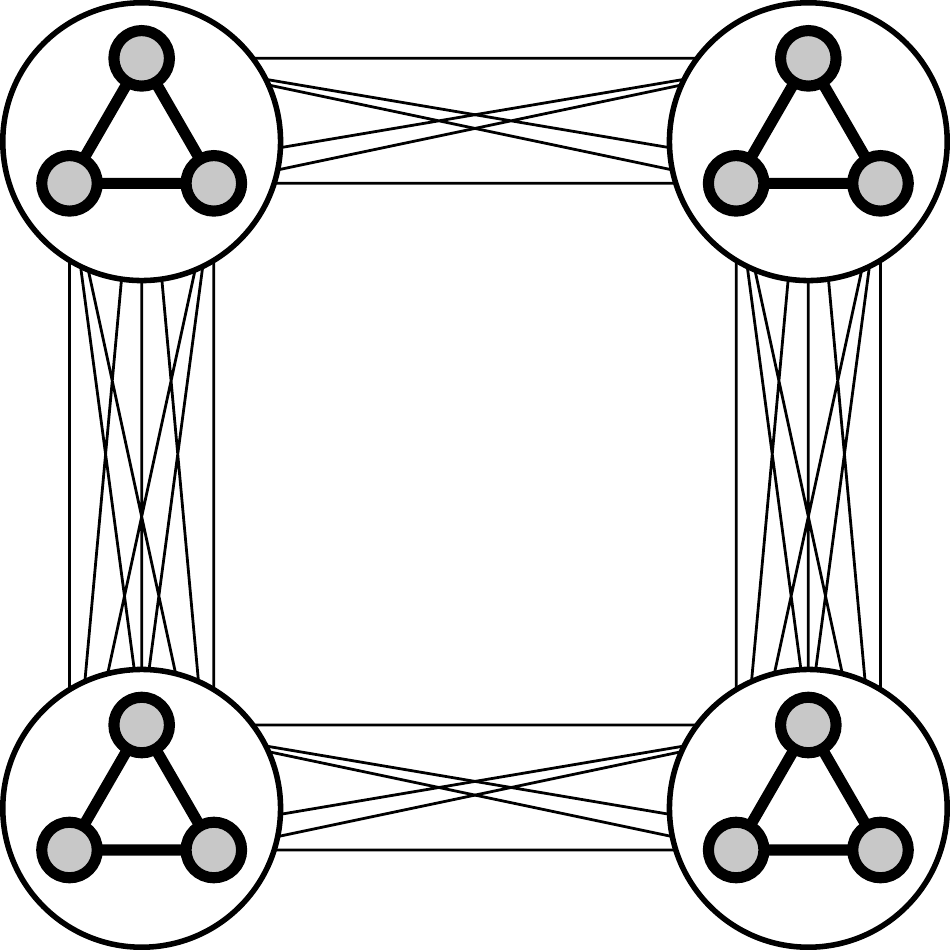}
\caption{A graph that does not admit a triangle decomposition.}\label{notridecomp}
\end{figure}
This is a graph on $4n$ vertices with $\delta(G) = 3n-1$.  Therefore, for any $\epsilon > 1/4$, there are instances of this graph for which every vertex has degree at least $(1-\epsilon)n$.  To see why no triangle decomposition is possible, partition the edges of $G$ into two kinds: the ``corner'' edges within each $K_n$, and the ``side'' edges connecting two different $K_n$'s.  Notice that any triangle that uses a side edge must use exactly two side edges and one corner edge in the construction above.  However, while there are $4n^2$ side edges, there are only  $4\binom{n}{2} = 2n^2 - 2n < 2n^2$ corner edges; so no such decomposition is possible.   

When Nash-Williams made his conjecture, it was not known if any value of $\epsilon$ existed at all.  Gustavsson \cite{Gustavsson_1991}, in a chapter of his 1991 doctoral thesis, proved that such a value of $\epsilon$ does exist for some values of $n$; specifically, he shows that any tridivisible graph $G$ on $n$ vertices with $\delta(G) \geq (1 - 10^{-24}) n$ admits a triangle decomposition, provided that $n \equiv 3, 9 $ mod $18$.  Garaschuk's 2014 dissertation considers a related problem: rational triangle decompositions of graphs, where a \textbf{rational triangle decomposition} is any way to write $G$ as the union of triangles, each triangle weighted by a nonnegative rational number, so that the sum over any edge of $G$ of all of the triangles containing that edge is 1.  Garaschuk proves in her thesis that any tridivisible graph $G$ on $n$ vertices with $\delta(G) \geq \frac{22}{23}n$ admits a rational triangle decomposition.

In 2015, Keevash \cite{keevash2014existence} proved a broader result: the existence conjecture for combinatorial designs.  He accomplishes this by showing that any hypergraph with an appropriate pseudorandom condition can be decomposed into cliques, using a novel technique dubbed Randomized Algebraic Construction.  As a special case, Keevash adapts \cite{keevash2015counting} these methods to the study of tridivisible graphs, and has proven the following:
\begin{thm}
Let $G$ be a graph on $n$ vertices.  Define its \textbf{density}, $d(G)$, as the ratio $|E(G)|/\binom{n}{2}$. We say that $G$ is \textbf{$c$-typical}, for some constant $c$, if every $v \in V(G)$ has $\deg(v)$ between $(1-c)d(G)n$ and $(1+c)d(G)n]$, and moreover any two vertices in $V(G)$ have somewhere between $(1-c)d(G)^2n$ and $(1+c)d(G)^2n$ neighbors in common.

There are constants $c_0, n_0$ such that any $c$-typical tridivisible graph on $n$ vertices admits a triangle decomposition, provided that $d(G) > n^{-10^{-7}}$, $n > n_0$, and $c < c_0 d(G)^{10^6}$.

\end{thm}
The constants here are worse than those of Gustavsson, as the techniques used in Keevash's paper rely on randomization and applications of Szemeredi regularity lemma-style results.  Rather, the power in Keevash's work is in its tremendous range.

More recently, Barber, K\"uhn, Lo and Osthus dramatically improved on this result in \cite{barber2016edge}, in which they develop a polynomial-time randomized algorithm that can find a triangulation of any graph on $n$ vertices with minimum degree at least $9n/10 + o(n)$ with high probability.  In fact, their paper derives similarly nice conditions under which any graph admits a $F$-decomposition, for any $F$, and is used by Bowditch and Dukes to show that any partial Latin square containing no more than $\frac{1}{25}$ entries in any row, column or symbol can be completed.

In this paper, we consider the subtask of finding such triangulations explicitly; i.e. via algorithms that are guaranteed to complete.

\begin{thm}
Take any $\epsilon < \frac{1}{432}$.  Then there is some constant $\xi > 0$ such that any tridivisible graph $G$ with $\delta(G) \geq (1- \epsilon)n$ and $|E(G)| \geq (1 -\xi)\binom{n}{2}$ has a triangle decomposition, and can find such a completion in $O(n^4)$ steps.
\end{thm}

\section{Techniques: Latin Squares, Trades, Overloading, and Local vs. Global Constraints}

Our proof techniques mirror those of Gustavsson as opposed to Keevash and Barber's later work, as we are seeking an explicit polynomial-time algorithm that is guaranteed to always complete.  We review these concepts here.

\begin{defn}\label{tradedefn}
Given a graph $G$, a \textbf{graph decomposition} $\mathcal{H}$ is a collection $\{H_1, \ldots H_k\}$ of subgraphs of $G$, such that the edges of $G$ are partitioned by these $H_i$'s. 

Let $G$ be a graph with associated decomposition $\mathcal{H} = \{H_1, \ldots H_k\}$.  Pick any subset of these subgraphs $\{H_1', \ldots H_l'\}$.  The union $ \bigcup_{i=1}^l H_i'$ of these subgraphs creates some subgraph $J$ of $G$, that may in turn have some other graph decomposition $\{H_1^\star, \ldots H_m^\star\}$.  If we take $\mathcal{H}$ and exchange the  $\{H_1', \ldots H_l'\}$ subgraphs for the $\{H_1^\star, \ldots H_m^\star\}$ subgraphs, this new collection is still a decomposition of $G$.  We call any such pair  $\{H_1', \ldots H_l'\}$, $\{H_1^\star, \ldots H_m^\star\}$ a \textbf{trade} on $(G, \mathcal{H})$. 

 For example, the trade below exchanges a hexagon and seven triangles for nine triangles:
\begin{center}
\includegraphics[width=3in]{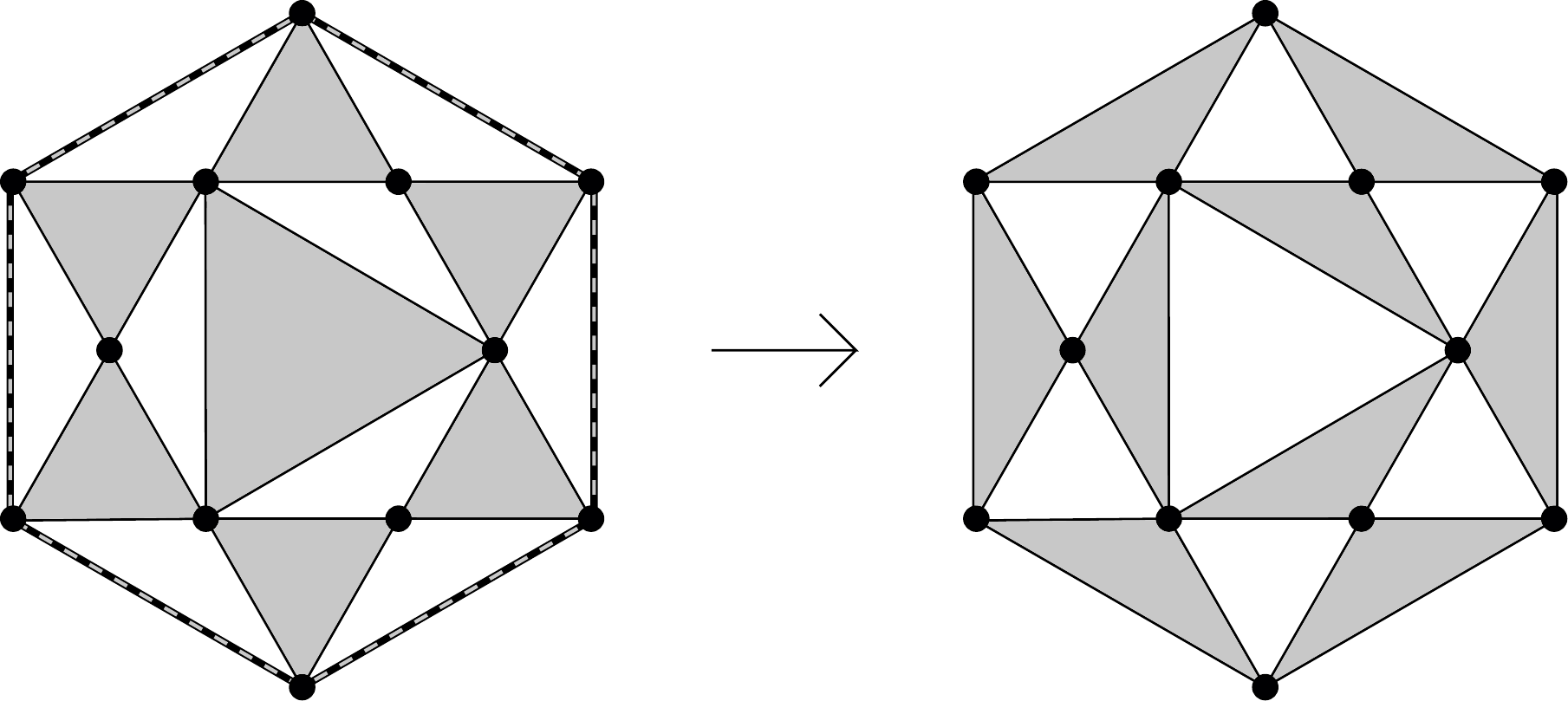}
\end{center}
\end{defn}

\begin{defn}
A \textbf{Latin square} of order $n$ is an $n \times n$ array filled with $n$ distinct symbols (typically, $\{1, \ldots n\}$), such that no symbol is repeated in any row or column.  A \textbf{partial Latin square} of order $n$ is simply an order-$n$ Latin square where we also allow cells to be blank.  We say that a partial Latin square $P$ is \textbf{completable}, and that $L$ is a \textbf{completion} of $P$, if the blank cells of $P$ can be filled with symbols in such a way that the resulting array $L$ is a Latin square.

A $n \times n$ partial Latin square $P$ is called \textbf{$\epsilon$-sparse} if no symbol occurs more than $\epsilon n$ times in $P$, and moreover no row or column contains more than $\epsilon n$ filled cells.
\end{defn}

Triangle decompositions of graphs are related to partial Latin squares.  To be precise: suppose that we have a tripartite\footnote{A \textbf{tripartite} graph $G = (V_1, V_2, V_3)$ is any graph whose vertices $V(G)$ can be partitioned into three sets $V_1, V_2, V_3$, such that no edge of $G$ has both endpoints within the same $V_i$.  Given any tripartite graph $G$ and any $v \in V_i$, we will let $\deg_+(v)$ denote the number of edges from $v$ to $V_{i+1}$ and $\deg_-(v)$ denote the number of edges from $v$ to $V_{i-1}$.  Note that the subscript labels are calculated mod 3 here; in general, our subscript labels are calculated with modular arithmetic where appropriate.} graph $G = (V_1, V_2, V_3)$, with $|V_i| = n$ for all $i$, that has a triangle decomposition.  If we think of $V_1$ as the ``rows,'' $V_2$ as the ``columns,'' and $V_3$ as the ``symbols,'' we can pair up each triangle in $G$'s decomposition with a triple $(r,c,s)$.  Because this is a triangle decomposition, no edge is used in more than one triangle; that is, no two triples $(r,c,s)$ agree at more than one place.  
\begin{center}
\includegraphics[width=3.5in]{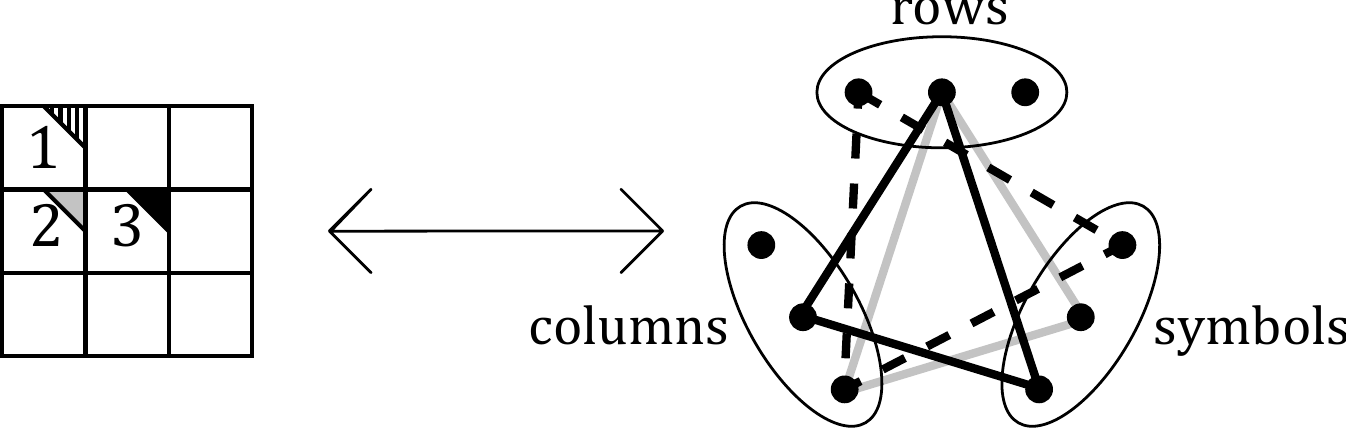}
\end{center}

Given these triples, we can form the corresponding $n \times n$ partial Latin square that has symbol $s$ in entry $(r,c)$ if and only if $(r,c,s)$ was one of our chosen triples; this is a partial Latin square by construction.  Moreover, this entire process is reversible; we can transform any partial Latin square into such a graph.

In particular, the transformation above turns any $ \epsilon$-sparse partial Latin square $P$ into a triangulated tripartite graph in which no vertex has more than $\epsilon n$-many neighbors in any one given part.  Therefore, triangulating the tripartite complement of this graph corresponds to finding a completion of $P$.

This correspondence suggests a connection between a conjecture of Daykin and Haggkvist\cite{Daykin_Haggkvist_1984} that all $ \frac{1}{4}$-dense partial Latin squares are completable and a tripartite version of the Nash-Williams conjecture.  As well, it suggests that results on completing partial Latin squares may be of use in finding triangle decompositions.

This link is the key to both Gustavsson's proof and ours.  To be specific: Gustavsson adapts results of Chetwynd and Haggkvist \cite{Chetwynd_Haggkvist_1985} to prove that all $\epsilon$-sparse $n \times n$ partial Latin squares are completable, for all $n \geq 10^7, \epsilon < 10^{-5}$, and uses this result along with a sequence of graph trades to triangulate graphs.

One of the authors of this paper in \cite{bartlett2013completions} came up with a strengthening of this Chetwynd-Haggkvist-Gustavsson result:
\begin{thm}\label{bartthm}
Any $ \epsilon_1$-sparse $n \times n$ partial Latin square containing no more than $\epsilon_2 n^2$ filled cells in total can be completed, if $\epsilon_1 < \frac{1}{12}, \epsilon_2 < \frac{ \left(1-12\epsilon_1\right)^{2}}{10409}$. Moreover, such a completion can be found via an algorithm with runtime $O(n^4)$.
\end{thm}
When $\epsilon_1 = \epsilon_2 \leq 9.8 \cdot 10^{-5}$, this is a simple strengthening of the earlier result.  Unlike before, however, we can now decouple the ``local'' constraint on how many entries are used in any row, column or symbol from the ``global'' constraint on how many cells are filled overall.  This lets us get results that are locally very close to the Daykin-Haggkvist conjecture while still allowing $O(n^2)$ many cells to be filled overall; for instance, Theorem \ref{bartthm} shows that any $1/13$-sparse $n \times n$ partial Latin square is completable if it contains at most $5 \cdot 10^{-7} \cdot n^2$ filled entries overall.  

This suggests a similar decoupling of local and global constraints may be possible when we go to triangulate graphs.  In other words, it is natural to think that constants $\epsilon_1, \epsilon_2$ exist such that any graph $G$ on $n$ vertices with $\delta(G) \geq (1-\epsilon_1)n, |E(G)| \geq (1-\epsilon_2)\binom{n}{2}$ admits a triangle decomposition.  

This is the primary goal of our paper.  Our proof methods parallel those of Gustavsson; we work with Latin squares and graph trades.  Given any tridivisible graph $G$, break it into ten parts; three parts $V_1, V_2, V_3$ all of equal size, each subdivided further into sets $V_{i,j}, i, j \in \{1,2,3\}$ all of equal size, along with a tenth part $V_{\textrm{rem}}$ consisting of ``leftover'' vertices.

Apart from these leftovers, this vertex partition induces a partition of $G$'s edges into thirteen parts.  There are nine almost-complete graphs, one corresponding to each $V_{i,j}$, as well as four tripartite graphs: the tripartite graph induced by $(V_1, V_2, V_3)$, as well as the three induced by $(V_{i,1}, V_{i,2}, V_{i,3})$.  We triangulate these edges as follows:

\begin{enumerate}
\item By finding an appropriate collection of matchings, we create edge-disjoint triangles that contain all of the edges incident to $V_{\textrm{rem}}$.
\item We then use Kirkman's theorem alongside recent results on rainbow matchings in edge-colored graphs to put almost all of the edges in each $V_{i,j}$ into triangles.  
\item We then use trades that turn triangles from our tripartite subgraphs into triangles that use up all of the ``leftover'' edges in our $V_{i,j}$'s, along with miscellaneous odds and ends.
\item We then use our correspondence between tripartite graphs and partial Latin squares to triangulate the edges in our four tripartite graphs.
\end{enumerate}

\subsection{Useful Notation}\label{notate}

Take any tridivisible graph $G$.  We introduce notation here that will help us find triangle decompositions of $G$.

First, for any such $G$, choose natural numbers $N, n,r$ such that $|V(G)| = N + r = 9n + r$, where $n$ is congruent to either one or three modulo six, and $r$ is no more than $35$.  Then, divide $G$'s vertices into ten parts; three sets $V_i, i \in \{1,2,3\}$ each of size 3n, each subdivided further into sets $V_{i,j}, i, j \in \{1,2,3\}$ of size $n$, along with a tenth part $V_{\textrm{rem}}$ of size $r$.   

For each pair $(i,j)$, let $G_{i,j}$ denote the graph induced by the vertex set $V_{i,j}$, and for each pair of pairs $(i,j)\neq (k,l)$, let $G_{(i,j), (k,l)}$ denote the bipartite graph induced by the bipartition $V_{i,j}, V_{k,l}$.  As well, for each $i$ let $T_i$ denote the tripartite graph induced by the tripartition $(V_{i,1}, V_{i,2}, V_{i,3})$, and $T$ denote the tripartite graph induced by the tripartition $(V_1, V_2, V_3)$.  Finally, let $L$ denote the subgraph of $G$ containing all of its vertices but none of its edges; we consider $L$ to consist of the ``leftover'' edges of $G$, and will add edges to $L$ as our proof progresses.

Under this vertex partition, the edges of $G$ have been partitioned into the distinct subgraphs $G_{i,j}, T_i, T, L$ described above.  We define families of constants $\epsilon, \xi$ that help us quantify precisely how ``almost-complete'' these parts are.  To be precise:
\begin{itemize}
\item If $H$ is a graph that we have explicitly described as a $m$-partite graph with $m$ parts $V_1, \ldots V_m$ of size $k$, define $\xi_H$ as the value of $\xi$ such that $E(H) = \binom{m}{2} k^2 - \xi k^2$.  As well, define $\delta(H)$ by
\begin{align*}
\delta(H) = \min_{v \in V(H)} \min_{j: v \notin V_j} \deg_j(v),
\end{align*}
and set $\epsilon_H$ equal to the value of $\epsilon$ such that $\delta(H) = (1-\epsilon)k.$
\item For any graph $H$ on $k$ vertices that we have not explicitly described as a $m$-partite graph of some kind, let $\xi_H$ denote the value of $\xi$ such that $E(H) = \binom{k}{2} - \xi k^2$.  Similarly, let $\epsilon_H$ denote the value of $\epsilon$ such that $\delta(H) = (1-\epsilon)k$.
\item Finally, for a graph $G$ and subgraph $H$ of $G$, let $\overline{G \setminus H}$ denote the complement of $G \setminus H$, and $\epsilon^c_H, \xi^c_H$ denote the values $\epsilon_{\overline{G \setminus H}}, \xi_{\overline{G \setminus H}}$ respectively. 
\end{itemize}

We think of these $\xi$ values as measuring how far parts of our graph are from a complete graph in a ``global'' sense, as they capture the total number of edges missing from our graph.   Similarly, we think of the $\epsilon$ values as measuring how far our graph is from a complete graph in a ``local'' sense: i.e. how many edges, at most, is any single vertex missing.  In particular, values of $\xi, \epsilon$ near zero correspond to almost-complete graphs in the senses above.

In the following sections, we show how to triangulate edges in this graph $G$ piece-by-piece, provided that each part is ``almost-complete'' in an appropriate sense (as captured by the constants defined above.)  

\subsection{Decomposing $V_{\textrm{rem}}$'s edges}  \label{remsection}
The first step in our proof is the simplest; we decompose the edges incident to $V_{\textrm{rem}}$ by finding a collection of appropriate matchings in our graph $G$.

Take any vertex $v \in V_{\textrm{rem}}$, and let $N(v)$ denote the collection of all vertices adjacent to $v$ in $G$. Arbitrarily partition $N(v)$ into two sets $N_1, N_2$ of equal size, and consider the bipartite graph induced by $(N_1, N_2)$, which we claim for small $\epsilon_G$ satisfies Hall's property \cite{Hall_1945}.

To see this, take any nonempty subset $S \subseteq N_i$. Notice that if $|S| \leq  \frac12 (|V(G)| - 3 \epsilon_G |V(G)|)$ we trivially have $|N(S)|  \geq |S|$, as $S$ contains a single vertex and the minimal  degree of any vertex in $(N_1, N_2)$ is at least $ \frac12 (|V(G)| - 3 \epsilon_G |V(G)|)$.  Conversely, if $|S| >   \frac12 (|V(G)| - 3 \epsilon_G |V(G)|)$ and $\epsilon_G < 1/6$ we have $|S| > \frac32 (\epsilon_G |V(G)|)$, and therefore that every $v \in N_{i+1}$ has at least one neighbor in $S$.  In particular, this means that $|N(S)| = |N_{i+1}| \geq |S|$; so $(N_1, N_2)$ satisfies Hall's property, and therefore we have a perfect matching in $(N_1, N_2)$.  

The union of these edges with all of the edges incident to $v$ forms a collection of $|N_i|$ edge-disjoint triangles.  Remove these triangles from $G$ along with $v$, and repeat this process until $V_{\textrm{rem}}$ is empty.  Because $V_{\textrm{rem}}$ contains at most 35 vertices, we can do this as long as $\epsilon_G \leq  \frac{1}{6} -\frac{70}{|V(G)|}$; this leaves us with a graph $G'$ in which $\epsilon_{G'} \leq \epsilon_G +  \frac{70}{|V(G)|}$ and $\xi_{G'} < \xi_G + \frac{105}{2|V(G)|}$.

Notice that in this graph, we have
\begin{align*}
\epsilon_{T'} \leq 3\epsilon_{G'}, \epsilon_{T_i'} \leq 9 \epsilon_G, \xi_{T'} \leq 9\xi_{G'}, \xi_{T_i'} \leq 81\xi_{G'}
\end{align*}
as in the worst-case scenario all of the edges missing from $G'$ are concentrated in one of the subgraphs $T, T_i$.


\subsection{Decomposing almost all $G_{i,j}$ edges}
The following lemma, when applied to any of our $G_{i,j}$ graphs, uses matchings and Kirkman's result to decompose almost all of its edges into triangles.
\begin{lem}\label{paddylem1}
Suppose that $H$ is a graph on $n$ vertices, where $n$ is congruent to 1 or 3 mod 6.  Then for any $\gamma > 0$, there is a subgraph $R$ of $H$, consisting entirely of edge-disjoint triangles, such that $\epsilon_R <  \epsilon_H + 4\sqrt{3\xi_H}$ and $|E(H \setminus R)| \leq (2\xi_H + \sqrt{3\xi_H})n^2.$

Moreover, such a subgraph can be constructed in $O(n^4)$ steps.
\end{lem}

\begin{proof}

Take the $n$ vertices $\{v_1, \ldots ,v_n\}$ of $H$, and construct the complete graph on these $n$ vertices; call this graph $K$.  As discussed earlier, it is known \cite{Kirkman_1847} that a triangle decomposition of $K$ is equivalent to a Steiner triple system and thus exists, as $n$ is congruent to 1 or 3 mod 6.  Let $R$ denote such a triangle decomposition of $K_n$.

Given this triangle decomposition, create a $n$-edge-coloring of $K$ as follows: for any $i \in \{1, \ldots n\}$ and any edge $\{x,y\}$, color this edge $i$ if and only if the triangle $\{v_i, x,y\}$ is a triangle in $R$.  This is a proper edge-coloring of $K$ using $n$ colors, one for each vertex; moreover each color $i$ corresponds to a perfect matching of the set $\{v_1, \ldots v_n\} \setminus \{v_i\}$.  This coloring induces a proper edge coloring of $H$, where we give each edge in $H$ its corresponding color in $K$.

Turn $R$ into a collection of edge-disjoint triangles in $H$ by simply removing from $R$ any triangle that is not a subgraph of $H$.  Consider the number of edges in $R$.  Because every triangle we've removed from $R$ corresponds to at least one edge not in $H$, there are at most $2 \xi_H n^2$ edges in total in $H$ that are missing from $R$.  

In some senses, this set $R$ is already an almost-triangulation in a global sense: it contains all but at most $2 \xi_H n^2$ of the edges it could have.  However, it is possible that some individual vertices have very low or zero degree in the subgraph $R$, and therefore that $R$ is not an almost-triangulation in a local sense.  We use matchings to fix this here, in a similar fashion to our work with $V_{\textrm{rem}}$.  For any vertex $v \in H$, consider the following process for creating a large rainbow\footnote{In an edge-colored graph, a rainbow matching is a matching in which no two edges share the same color.} matching in $N(v)$ :
\begin{enumerate}
\item[(0)] Set $M = \emptyset,$ and  $U$ equal to the graph induced by $N(v)$; let $C_M$ denote the colors of edges currently in the matching $M$ and $C_U$ the colors of edges not currently used in $M$.
\item If there is an edge in $U$ whose color is in $C_U$, place that edge in $M$ and then delete it and its endpoints from $U$.  Repeat this process until no such edges exist; this takes $O(n^3)$ steps, as there are $O(n^2)$ edges in $U$ to check and we can perform at most $O(n)$ such transfers before $U$ is empty.
\item If there are two vertices $u, v \in U$ and an edge $\{x,y\} \in M$ such that $\{u,x\}, \{y,v\}$ are both colored with distinct colors from $C_U$, delete the edge $\{x,y\}$ from $U$ and add the edges $\{u,x\}$ and  $\{y,v\}$ to $M$.  Again, repeat this process until no such configurations exist; this takes $O(n^3)$ steps, as we have $O(n)$ edges in $M$ to check, each has $O(n)$ endpoints to examine, and we can perform $O(n)$ exchanges before $U$ is empty.

\end{enumerate}

Notice that for any edge $e$ in $M$, if both endpoints of $e$ are incident to $C_U$-colored edges connected to $U$ and there are at least five such edges, then we can always find a (2)-configuration.  Consequently, when this algorithm finishes, any edge $e$ in $M$ is incident to at most $\max(4, |U|)$ $C_U$-colored edges incident to $U$.   In particular, there are at most $(\max( 2, \frac{|U|}2 )) \cdot  (n-|U|)$ $C_U$-colored edges from $M$ to $U$.

Conversely, at the end of this process every vertex in $U$ has at least $n - \epsilon_H n - \frac12(n-|U|)$ $C_U$-colored edges to $M$, and so the total number of $C_U$-edges leaving $U$ to $M$ is at least $|U| \cdot (n - \epsilon_H n - \frac12(n-|U|))$.  Therefore, we either have $|U| \leq 4$ or

\begin{align*}
|U| \cdot \frac{n-|U|}{2} \geq |U| (n - \epsilon_H n - \frac12(n-|U|))  \quad \Rightarrow \quad |U| \leq \epsilon_H n.
\end{align*}

Choose the $k-1$ vertices $v_1, \ldots v_{k-1}$ with the largest differences between $\deg_H(v_i)$ and $\deg_R(v_i)$.  One by one, perform the following steps: find such a rainbow matching $M$.  Delete all of the triangles in $R$ that use edges in this rainbow matching, along with all triangles in $R$ incident to $v_i$; then, add to $R$ all of the triangles formed by combining our rainbow matching with all of the edges incident to $v_i$.  This process ensures that $\deg_H(v_i) - \deg_R(v_i) \leq \max(4, \epsilon_H n)$.  As well, for any other vertex $w \in R$, we have decreased $\deg_R(w)$ by at most six; we deleted at most one triangle containing $w$ incident to $v$, at most one triangle containing an edge $e$ from $M$ where $w$ is an endpoint of $e$, and at most one triangle containing an edge $e$ from $M$ where $e$ is colored $w$.  Finally, each run of this process decreases the total number of edges in $R$ by at most $3n$, as we deleted at most $n +n-\frac{|U|}{2}$ triangles and added back in $n-\frac{|U|}{2}$ triangles.

Consider the vertex $v_{k}$ whose difference $\deg_H(v_{k}) - \deg_R(v_{k})$ was the $k$-th largest at the start of this process; in the worst-case scenario, this difference was equally as large as all of our other differences, and it grew by six after each run of our algorithm.  In particular, this means that 
\begin{align*}
\deg_H(v_{k}) - \deg_R(v_{k}) \leq 6(k-1) + \frac{2 \xi_H n}{k}.
\end{align*}
Conversely, for any of the vertices $v_i$ our algorithm ran on, in the worst-case scenario, we have
\begin{align*}
\deg_H(v_{i}) - \deg_R(v_{i}) \leq 6(k-1) + \max(\epsilon_H n, 4).
\end{align*}
If we let $k = \lceil \gamma n \rceil $ and combine these results, this gives us that 
\begin{align*}
\epsilon_R < \max\left(   \epsilon_H +  6\gamma + \frac{2 \xi_H}{\gamma}, 6\gamma+   \max\left(\epsilon_H, \frac4n\right)   \right) 
\end{align*}
Notice that we can assume $\frac{1}{n} \leq \sqrt{\xi_H}, \frac{1}{n} \leq \epsilon_H$, as otherwise we would have $\xi_H n^2 < 1$ or $\epsilon n < 1$, which in particular implies that $H$ is complete and trivializes our problem.  Consequently, setting $\gamma = \sqrt{\frac{\xi_H}{3}}$ optimizes the bound above and yields
\begin{align*}
\epsilon_R < \epsilon_H + 4\sqrt{3\xi_H}.
\end{align*}

As well, because we start with $|E(H \setminus R)| \leq \xi_H n^2$, delete $3n$ edges on each run of the process above, and repeat this process $\lceil \gamma n\rceil -1$ times, we have
\begin{align*}
|E(H \setminus R)| <  (2\xi_H + 3\gamma)n^2 \leq (2\xi_H + \sqrt{3\xi_H})n^2.
\end{align*}

Finally, it bears noting that we can find this $R$ in $O(n^4)$ steps, as we need to perform only $\gamma n$ iterations of an $O(n^3)$-step process.
\end{proof}

Take our graph $G'$, apply Lemma \ref{paddylem1} above to each $G_{i,j}'$ subgraph, and remove the resulting sets of triangles from $G'$.  Let $G^{(2)}$ denote the resulting graph; within $G^{(2)}$, let $L^{(2)}$ contain the leftover edges from the nine $G_{i,j}$ graphs. 

The process above increased the value of each $\epsilon_{G_{i,j}}$ by at most $4\sqrt{3\xi_{G'}}$; accordingly, we have 
\begin{align*}
\delta(L^{(2)}) < \epsilon_{G'}N + 9\cdot 4\sqrt{3\xi_{G'}}) n = (\epsilon_{G'} + 4\sqrt{3\xi_{G'}})N.
\end{align*}

As well, each run of Lemma \ref{paddylem1} leaves us with at most $(2\xi_{G_{i,j}'} + \sqrt{3\xi_{G_{i,j}'}})n^2$ edges in each $G_{i,j}^{(2)}$. The total number of edges in $L$ is thus maximized when the edges in $\overline{G'}$ are equally distributed over the graphs $\overline{G_{i,j}}$, in which case we have
\begin{align*}
|E(L^{(2)})| < \left(2\xi_{G'} + \sqrt{\frac{\xi_{G'}}{3}}\right)N^2.
\end{align*}

In other words,  we have put ``most'' of the edges in each $G_{i,j}$ into triangles.  Note that this process does not affect the $T,T_i$ subgraphs.

\subsection{Ensuring out-degree equals in-degree in our tripartite subgraphs}
The next part helps clean up our tripartite subgraphs for later work:
\begin{lem}\label{landonlem}
Take any tripartite graph $R$ on $3n$ vertices with a tripartition $V_1, V_2, V_3$, where $|V_1| = |V_2| = |V_3| = n$.  If $\epsilon_R < \frac{1}{12}, \xi_R < \frac{\epsilon_R}{6}$, then we can delete at most $ 6\xi_R n^2$ edges from $R$ in such a way that $\deg_+(v) = \deg_-(v)$ for each vertex in $R$, without changing $\epsilon_R$ or removing more than $3\xi_R n$ edges from any vertex.  Moreover, these deletions can be performed in $O(n^3)$ steps.
\end{lem}
\begin{proof}
Notice that because $|E(R)| = \sum_{v \in R}\deg_+(v) = \sum_{v \in R}\deg_-(v)$, the existence of a vertex $x \in R$ such that $\deg_+(x) > \deg_-(x)$ is equivalent to the existence of a vertex $y \in R$ such that $\deg_+(y) < \deg_-(y)$.  

Based on this observation, we describe a process that takes in any two such vertices $x, y \in V(G)$ with $\deg_+(x) > \deg_-(x)$, $\deg_-(y) > \deg_+(y)$, creates a positively-oriented path (i.e. one in which edges go from $V_i$'s to $V_{i+1}$'s) on at most four edges from $x$ to $y$, and then deletes this path from $G$.  Doing this decreases the gap between $\deg_+(x), \deg_-(x)$ and $\deg_+(y), \deg_-(y)$ by one each, and does not change the difference between $\deg_+, \deg_-$ for any other vertex. 

To do this, take any such $x,y \in V(G)$ with $\deg_+(x) > \deg_-(x)$, $\deg_-(y) > \deg_+(y)$.  Construct the appropriate positively-oriented path from $x$ to $y$ (illustrated below) vertex-by-vertex.
\begin{figure}[H]
\includegraphics[width=4in]{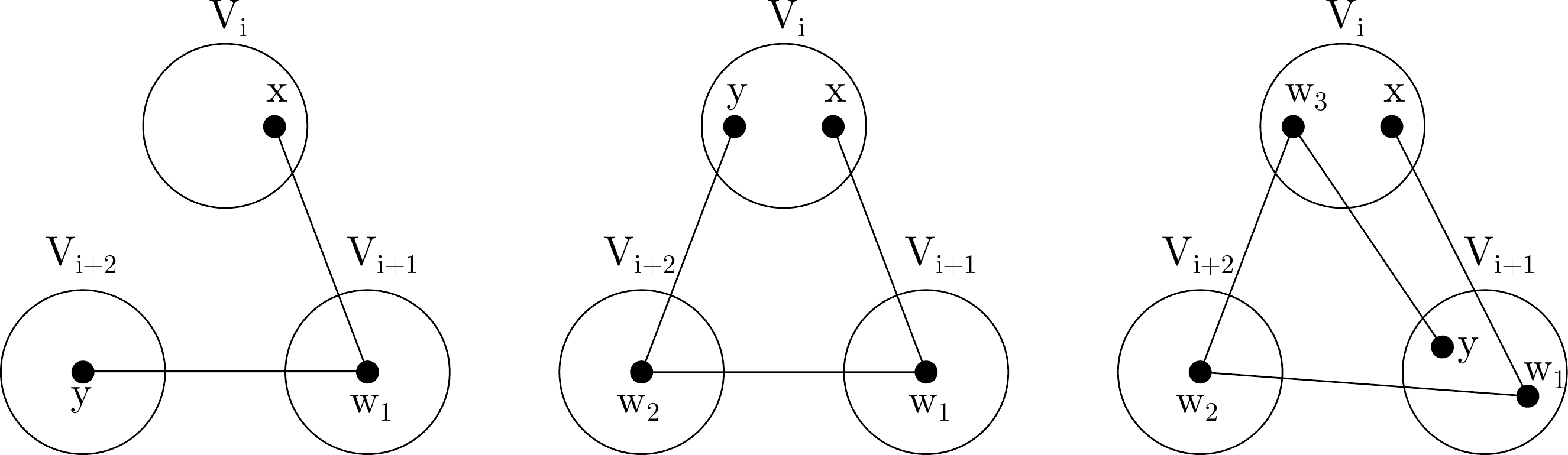}
\caption{Three possible paths.}\label{landonpaths}
\end{figure}

Consider the third case: here, we have at least $n-\epsilon_R n$ choices for $w_1$ adjacent to $x$.  Amongst these choices, pick the one for which $\min(\deg_+, \deg_-)$ is the highest.  Similarly, there are $n - \epsilon_R n$ choices of $w_2$ adjacent to $w_1$, and $n - 2\epsilon_R n $ choices of $w_3$ adjacent to $w_2, y$; again, choose the highest-degree vertices possible when making these choices.  Similar arguments construct paths for the first two cases; in all three cases, because $n - 2\epsilon_R n \geq \frac{5}{6}n \geq  1$ throughout this entire process we can always make such choices.

Do so, discard the resulting path from $R$, and repeat this process until $\deg_+(v) = \deg_-(v)$ for every vertex $v$; call the resulting graph $R'$.  Because there are at most $3 \xi_R n^2$ edges missing from $R$ in total and each run of the process above reduces the overall differences in degree by two, we need to run our algorithm at most $\frac{3}{2} \xi_R n^2$ many times.  Each time we run this process we are decreasing the value $\min(\deg_+, \deg_-)$ of any vertex by at most 1, and are doing so to at most one vertex in any part $V_i$ whose degree is the highest amongst our possible choices.

If $\xi_R < \frac{\epsilon_R}{6}$, then at most a third of the vertices in $R$ can be missing more than $\frac{\epsilon_R}{2}$ edges.  If we consider these choices to be ``bad'' when we pick our vertices $w_i$ and additonally assume that $\epsilon_R < \frac{1}{12}$, then there are at least $\frac{n}{2} $ not-bad choices available to us at our first step.  Choosing the vertex with highest degree ensures that we make such a choice. Moreover, as we keep choosing these vertices we will never wind up changing the value of $\epsilon_R$; in the worst-case scenario always choosing the highest-available degree vertex is still a better strategy than distributing our $\frac{3}{2} \xi_R n^2$ choices evenly over the $\frac{n}{2} $ available not-bad choices, which decreases the degree of each not-bad vertex by at most $3\xi_R n < \frac{\epsilon}{2}$, ensuring that $\epsilon_{R'} = \epsilon_R$.

At most $6\xi_R n^2$ edges are deleted in the process above, as we construct at most $\frac32 \xi_R n^2$ paths throughout this proof and each has length 4.   Finally, we have $O(n)$ steps to choose appropriate vertices $w_i$ for any given walk and $O(n^2)$ walks to construct in total, giving us $O(n^3)$ runtime as claimed.
\end{proof}

Apply this lemma to each of the four tripartite subgraphs $T^{(2)}, T_{i,j}^{(2)}$ of $G^{(2)}$, removing the deleted edges from each $T^{(2)}, T_i^{(2)}$ and adding them to the subgraph $L^{(2)}$; let $G^{(3)}$ denote the resulting graph.  In the worst-case scenario where all of the edges in $\overline{G}$ are in exactly one of the $ \overline{T_i^{(2)}}$ graphs,  we can see that $\delta(L^{(2)})$ increases by at most $27\xi_{G'}N$ and $|E(L^{(2)})|$ increases by at most $6\xi_{G'}N^2$.  Therefore, we have
\begin{align*}
\delta(L^{(3)}) \leq (\epsilon_{G'} + 4\sqrt{3\xi_{G'}} + 27\xi_{G'})N, \qquad |E(L^{(2)})| < \left(8\xi_{G'} + \sqrt{\frac{\xi_{G'}}{3}}\right)N^2.
\end{align*}

As well, we have
\begin{align*}
\epsilon_{T^{(2)}} \leq \epsilon_{T'}, \qquad \epsilon_{T_i^{(2)}} \leq \epsilon_{T_i'}, \qquad \xi_{T^{(2)}} \leq 7 \xi_{T'}, \qquad  \xi_{T_i^{(2)}} \leq 7\xi_{T_i'}.
\end{align*}
Finally, notice that the subgraphs $T^{(2)}, T_{i}^{(2)}$ are also now tridivisible.  This is not hard to see: take any tripartite graph with the property that $\deg_+(v) = \deg_-(v)$ for every vertex $v$.  We can immediately observe that this forces $\deg(v)$ to be a multiple of two; as well, noting that the sum of edges leaving $V_i$ to $V_{i+1}$ is the same as the sum of edges entering $V_{i+1}$ from $V_i$ tells us that the total number of edges in our graph is a multiple of three.

\subsection{Decomposing all ``leftover'' edges into triangles}
Here, we use trades to perform the trickiest work in this paper; decomposing the leftover edges $L$ into a set of triangles by using our tripartite graphs $T, T_i$.  

\begin{lem}\label{pathlem}(Leftover lemma.)
Take any tridivisible graph $G$ on $9n$ vertices, with subgraphs $T, T_i, L$ and vertex partitions $V_i, V_{i,j}$ as described in subsection \ref{notate}.  Let $H = G \setminus L$, and suppose that $H$ is tridivisible and 9-partite with respect to the partition $V_{i,j}$.  

Suppose that 
\begin{align*}
n - 4 \max ( \epsilon_{T_1}, \epsilon_{T_2},\epsilon_{T_3}, \epsilon_T + \frac12 \delta(L)) n - 4\delta(L)n \geq 3.
\end{align*}

Then, there is some collection $R$ of edge-disjoint triangles from $H$ that satisfies the following properties:
\begin{enumerate}
\item Each triangle in $R$ is either a subgraph of $T$, or of $T_i$ for some $i$.  
\item $L \cup R$ can be written as an edge-disjoint union of triangles.
\item $\epsilon_{T_i \setminus R} \leq \epsilon_{T_i} + \frac{\delta(L)}{n}, \epsilon_{T \setminus R} \leq  \epsilon_{T} + \frac{\delta(L)}{2n}.$
\item $\xi_{T_i \setminus R} \leq \xi_{T_i} + \frac{9|E(L)|}{2n^2},  \xi_{T \setminus R} \leq \xi_{T} + \frac{71|E(L)|}{36n^2}.$
\end{enumerate}
\end{lem}
\begin{proof}
Note that because $G, H$ are both tridivisible, so is $G\setminus H = L$; in particular, this means that the degree of every vertex in $L$ is even.  Consequently, the edges of $L$ can be decomposed into a collection of edge-disjoint cycles.  Let $\mathcal{C}$ denote this set of cycles.  



Our lemma proceeds in two stages:
\begin{itemize}
\item \textbf{Shrinking}: we describe a process that takes in a cycle $C_n$ and triangles from the graphs $T, T_i$, and outputs a cycle with length $n-3$ along with more triangles.
\item \textbf{Merging}: we describe a second process that lets us merge any two cycle $C_a, C_b$ in $\mathcal{C}$ into a collection of cycles whose lengths are all congruent to either 0 or $a+b$ mod 3.  These trades only use triangles from the graph $T$.
\end{itemize}

Appropriate application of these processes to the set $\mathcal{C}$ will reduce it to a collection of triangles.  If we are careful in selecting and applying these trades, the collection of all triangles $R$ used in these trades will have the properties claimed by our lemma, and complete our proof.

We start by describing how to shrink a $C_6$ to a $C_3$.  Take any such $C_6$ from $\mathcal{C}$, and let $w_1, w_2, w_3, w_4, w_5, w_6$ be the consecutive vertices of this $C_6$.  Label each of our vertices $w_j$ with an $i$ whenever $w_j \in V_i$.  Enumerating the collection of all such labellings of hexagons, up to rotation, reflection, and permuting the labels $\{1,2,3\}$, is a straightforward application of Burnside's lemma, and tells us that there are 22 such possible distinct labelings.
We draw 22 diagrams corresponding to these labellings in Figure \ref{22hex}, along with a collection of labeled triangles for each configuration.  Some of the vertices in our diagrams are colored; these will be used to further associate vertices with sets $V_{i,j}$ when appropriate.

\begin{figure}
\includegraphics[width=\textwidth]{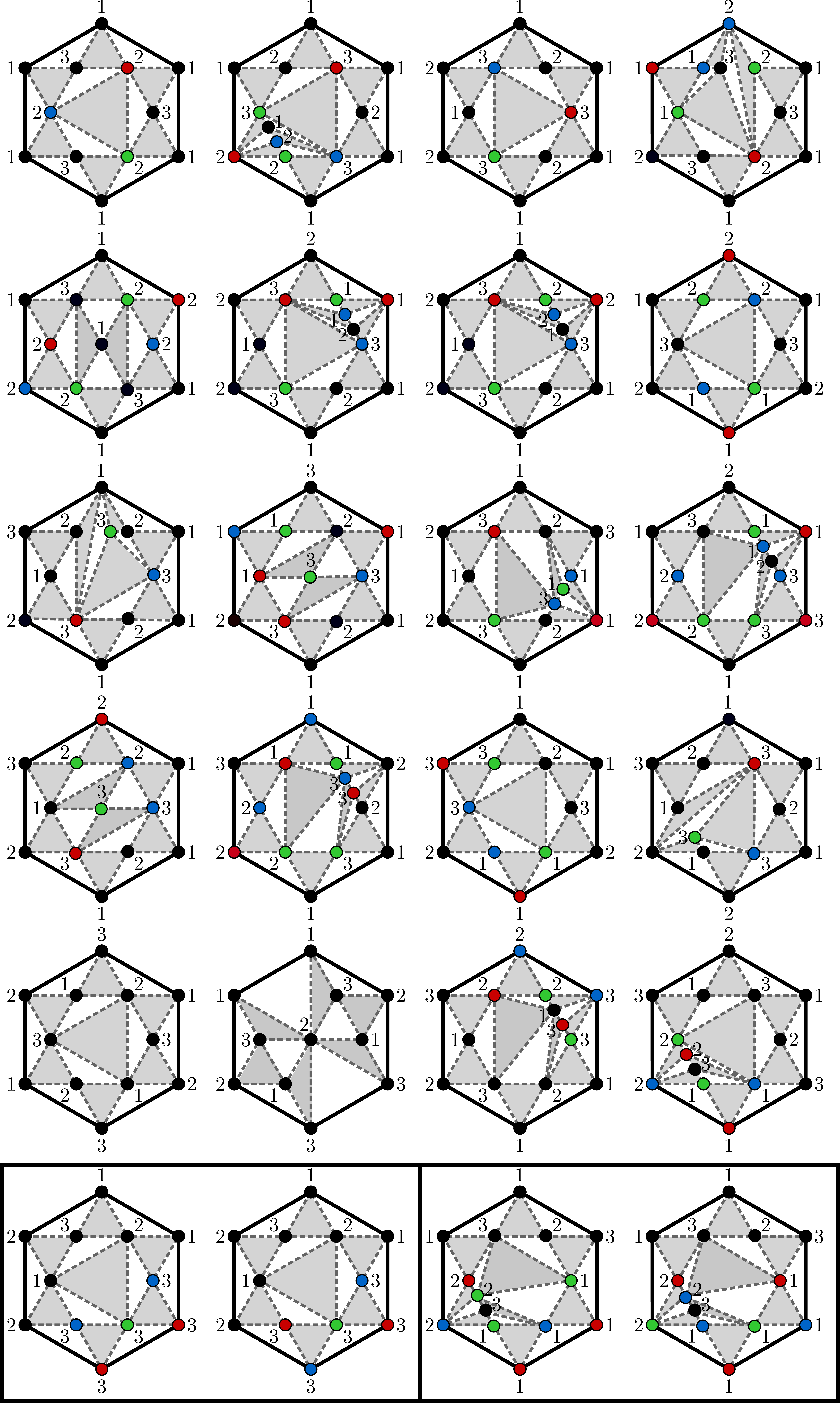} 
\caption{Labeled hexagons and their associated trades. Vertex labels match vertices to sets $V_i$.  Colors can be used to further partition vertices into sets $V_{i,j}$ where needed; two hexagon labelings are repeated (at bottom) to let us deal with any such partition.}
\label{22hex} 
\end{figure}

Take our $C_6$.  By permuting the labels $1,2,3$ if needed, find the appropriate diagram in Figure \ref{22hex} that shares the same labeling as our hexagon.  Replace the boundary of this diagram with our $C_6$, and then use the fact that each vertex in $C_6$ belongs to a set $V_{i,j}$ to refine the label of each such boundary vertex to a pair $(i,j)$.

We now subdivide the labels of our internal vertices into pairs $(i,j)$ so that no two adjacent vertices are labeled with the same ordered pair.  For any $i$, look at the subgraph of our diagram's vertices given by only considering vertices labeled $i$.  Take any connected component of this subgraph.  In the case that this connected component has just one internal vertex, any refinement of $v$'s label will satisfy our desired property.  In all other cases, our diagram has given a 3-coloring of the vertices in each such component; use each such coloring along with the labelings of our exterior vertices to further refine each label to a pair $(i,j)$.  Where possible, choose bijections that minimize the maximum number of times any pair $(i,j)$ is used in this diagram.  If we associate each $(i,j)$-labeled vertex in our diagram with the set $V_{i,j}$, note that each triangle shaded in gray is a subgraph of one of our tripartite graphs $T, T_1, T_2, T_3$.  Exchanging our hexagon and the gray triangles in this diagram for the white triangles is therefore a trade that accomplishes our desired shrinking effect.

If we started with a cycle $C_n$ with $n \geq 7$, our process is similar.  Take any path of length six that is a subgraph of our $C_n$.   Pretend that the ends of this $P_6$ are actually joined up to form a hexagon, and look up the appropriate diagram in Figure \ref{22hex}; now, remove the edge between $w_1$ and $w_6$ in the Figure \ref{22hex} trade.  This yields a trade that takes in our path $P_6$ and some number of triangles, and outputs a smaller path $P_{3}$ along with one more triangle than we started with; in other words, we have a trade that shrinks the length of our cycle $C_n$ by three, as desired.
\begin{figure}[H]
\includegraphics[width=0.6\textwidth]{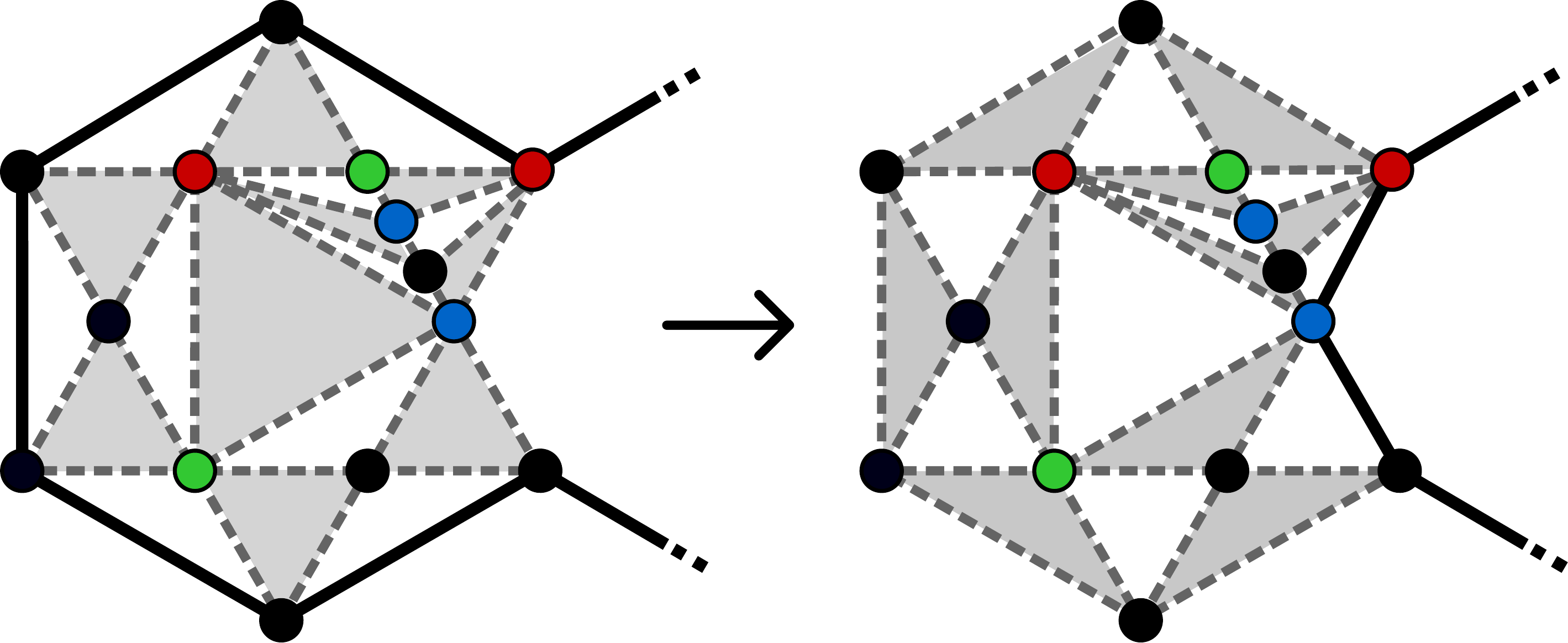} 
\caption{Using a hexagon trade to shrink a path by three edges.}
\end{figure}

As in Lemma \ref{landonlem}, create such a configuration by choosing internal vertices one-by-one with the desired edges.    Whenever we make such a choice, always choose the vertex with highest $H$-degree.

One-by-one, take each cycle in $\mathcal{C}$ whose length is a multiple of three, and run the process above on it until only triangles remain.  If the process here is ran more than once on a given cycle $C$, ensure that the $P_6$ chosen to start each subsequent run of this process contains the $P_3$ generated by the trade immediately prior, and that this $P_6$'s endpoints do not overlap with this $P_3$.  Note that doing this ensures that if $v$ was on the $P_3$ from an earlier step, it cannot occur in the $P_3$ in the next step.  Therefore, each instance of $v$ in a cycle in $L$ induces at most two occurrences of $v$ in the boundary of our trades thus far.

This leaves us with a set $\mathcal{C}$ in which all of our cycles have lengths that are not multiples of three.  Take any two cycles $C_a, C_b \in \mathcal{C}$; we describe a set of trades that merge these subgraphs into cycles whose lengths are all congruent to either 0 or $a+b$ mod 3.  

If $C_a, C_b$ share any vertices in common, choose one and call it $w$.  If $C_a$ has length longer than seven, repeatedly apply the shrinking trades from earlier to reduce it to a cycle whose length is either four or five.  Do so such that $w$ is never used in any of these shrinking trades, and so that the $P_3$ resulting from the last trade is adjacent to $w$, if such a vertex $w$ exists.  Do the same to $C_b$ if needed; as well, when doing this, ensure that the middle vertex in the $P_3$ in the last of these shrinking does not show up in any of the vertices in our now-shrunk $C_a$.  

Finally, suppose that $C_b$ already has length 4 or 5 but shares two or more vertices in common with $C_a$.  
\begin{figure}[H]
\includegraphics[width=\textwidth]{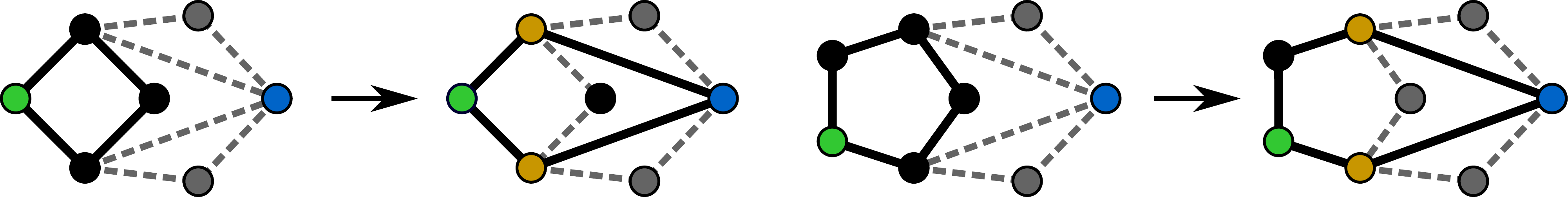} 
\caption{What to do with a $C_4$ or a $C_5$.  The vertex in green is the vertex $w$ that $C_a, C_b$ share in common, if it exists.}\label{mergec4c5}
\end{figure}
Choose three vertices one-by-one (colored gray and blue) that are part of two dashed-line triangles in $T$ as drawn in Figure \ref{mergec4c5}.  When doing this, ensure that the blue vertex is not in the cycle $C_a$, and always choose the vertex with highest $H$-degree possible when picking these vertices.  Perform the trades indicated below, and then use our shrinking lemma to reduce the $C_6$'s to triangles; this leaves us with a cycle with the same length as $C_b$.

If at the end of this process both $C_a, C_b$ are pentagons and they share at least two vertices in common, then casework shows that their union can be rewritten as either a $C_4 \cup C_6$, $C_3\cup C_7$, or $C_3 \cup C_3\cup C_4$, all of which are cycles whose lengths are congruent to 0 or $a+b$ mod 3.  Conversely, if we make sure to always shrink any cycle whose length is congruent to 1 mod 4 last, our two cycles $C_a, C_b$ will share at most one vertex in common if either $C_a, C_b$ had length congruent to 1 mod 4.  As a result, the only cases that remain are when we've shrunk $C_a, C_b$ to $C_4$'s and $C_5$'s that share at most a single vertex $w$ in common.

Using Figure \ref{mergefignew}, look up the appropriate figure, and use it to choose gray vertices one-by-one such that the dashed-line configurations drawn below are subgraphs of one of $T$. Again, maximize the $H$-degree of these chosen vertices where possible.  Performing the trades indicated leaves us with cycles whose lengths are all either 0 or $a+b$ mod 3, as desired. 

One-by-one, take any two cycles $C_a, C_b \in \mathcal{C}$ such that $a+b$ is congruent to 0 mod 3, perform the merging trades indicated above, and then use the shrinking trades from earlier to decompose the resulting cycles into triangles.  When doing this, shrink the resulting cycles so that none of the black or gold-colored vertices in this diagram are involved in more than one shrinking trade per cycle that they are in.  

With this done, repeatedly choose any triple $C_a, C_b, C_c$ for which $a,b,c$ are all congruent mod 3, merge $C_a$ and $C_b$, and shrink the resulting cycles so that the only cycle that remains has length less than six and is congruent to $a+b$ mod 3; call this cycle $C_d$.  When doing this, again ensure that none of the black or gold-colored vertices in this diagram are involved in more than one shrinking trade per cycle that they are in, and also that all of the vertices in $C_d$ are gray or blue.  Merge this $C_d$ with $C_c$, and shrink the resulting cycles as described earlier.   
\begin{figure}[H]
\includegraphics[width=\textwidth]{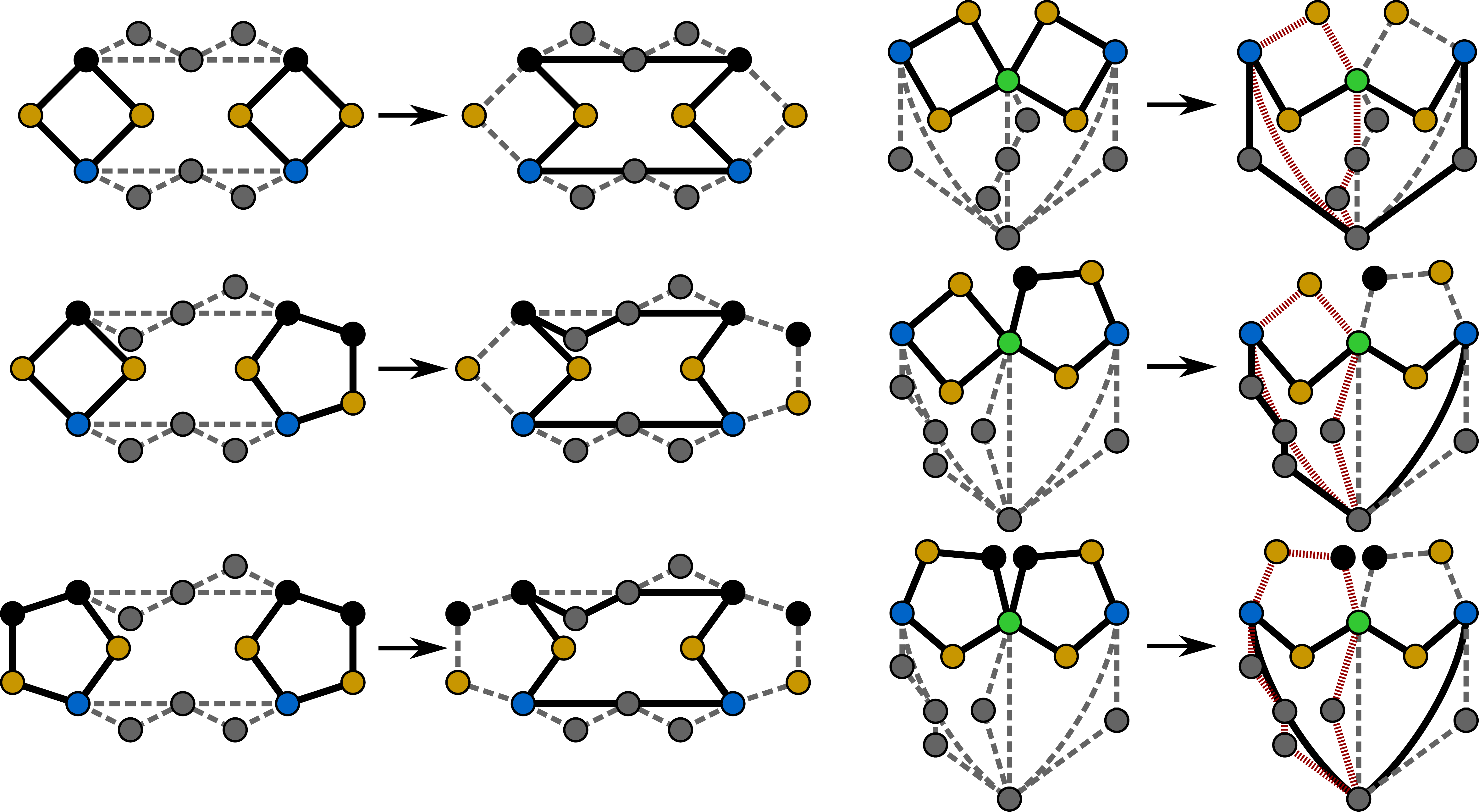} 
\caption{How to merge small cycles.   The green vertex is the vertex $w$ both cycles share in common, if it exists.  The gold vertices in the cycles being merged correspond to either the endpoints of the $P_3$ created by the last shrinking trade if we applied a shrinking trade to that vertex, or correspond to the gold vertices in Figure \ref{mergec4c5} if we applied one of those trades.  The blue vertex is the vertex we got to choose in either that shrinking trade or Figure \ref{mergec4c5} trade.}\label{mergefignew}
\end{figure}

At the end of this process, we have used a collection of trades to completely decompose $\mathcal{C}$ into triangles, in which each trade consisted of triangles from the tripartite subgraphs $T, T_i$.  We finish our proof by determining the conditions needed to ensure these trades all exist. When choosing a vertex $v$ for any of these trades, we eliminate choices for the following reasons:
\begin{itemize}
\item The vertex $v$ is already in use somewhere else in the configuration.  As no internal vertex is chosen from a $V_{i,j}$ containing more than two other vertices in any of the diagrams presented in this proof, this eliminate at most two choices.
\item The vertex $v$ is not adjacent to all of its neighbors.  If we let $\epsilon = \max(\epsilon_{T_1 \setminus R}, \epsilon_{T_2\setminus R}, \epsilon_{T_3\setminus R})$, this eliminates at most $l\epsilon n $ choices, where $l$ is the number of neighbors of $v$ we have already chosen in this diagram.
\end{itemize}

If we order and select vertices in descending order according to their degree, examining Figures \ref{22hex}, \ref{mergec4c5}, \ref{mergefignew} shows that $l$ is at most four.  Therefore, we can make these all of these choices provided that $n - 4\epsilon n - 2 \geq 1$.  

To determine $\epsilon$, we turn our attention to the set $R$.  There are two different ways in which a vertex occurs in the trades that make up $R$:
\begin{itemize}
\item As a vertex that was part of a cycle in $\mathcal{C}$, and as such that we did not choose. 
\item As a vertex that we've chosen; i.e. as one of the internal vertices in our shrinking trades, or as one of the gray or blue vertices in our merge trades.
\end{itemize}

Notice that the least-efficient cycles for our trades are disjoint $C_4$'s; to deal with the twelve edges in three disjoint $C_4$'s we need to perform two merge trades and nine shrink trades to convert these cycles into triangles.  Examining Figures \ref{22hex}, \ref{mergefignew} shows us that each shrink and merge trade chooses at most four vertices from any $V_i, V_{i,j}$; as well, each time we choose a vertex $w \in V_{i,j}$ in either a shrink or merge trade, we use at most three edges from $w$ to any $V_{k}$ or  $V_{i,l}$. 

We make at most $\frac{11}{12}|E(L)|$ trades in total, and always choose the vertex with highest degree amongst our at least $n - 4\epsilon n - 2$ choices.  The worst-case scenario for $\epsilon$, then, is if we somehow had to keep repeatedly making these choices amongst the same set of $n - 4\epsilon n - 2$ vertices in some $V_{i,j}$, all of which had the same degree in $T_i$ and $T$.  In this setting, all of these vertices would have degree at least $n -\frac{\xi_{T_i} n^2}{n-4\epsilon n - 2}$ in $T_i$, $3n - \frac{\xi_T 9n^2}{n-4\epsilon - 2}$ in $T$ and we would decrease these degrees by at most $\frac{\frac{11}{12}\cdot 4 \cdot 3 \cdot |E(L)|}{n-4\epsilon n - 2}$.  Accordingly, if $\epsilon < \frac{1}{9}$ and
\begin{align*}
2\xi_{T_i} + 22\frac{|E(L)|}{n} < \epsilon_{T_i}, 18\xi_{T} + 22|E(L)| < \epsilon_{T}
\end{align*}
these chosen vertices do not change the constants $\epsilon_{T_i}, \epsilon_T$.

As a result, the only vertices that affect the values $\epsilon_{T_i}, \epsilon_T$ are those that we do not choose; i.e. those that came from cycles in $\mathcal{C}$.  Each occurrence of a vertex $v$ in a cycle in $\mathcal{C}$ induces at most two occurrences of $v$ as an external vertex in a shrinking trade.  If $v \in V_{i,j}$, examining Figures \ref{22hex}, \ref{mergefignew} shows us that for any $k \neq i, l \neq j$ we use at most one edge from $v$ to any $V_{i,l}$ or  $V_k$ in any such shrinking trade, and at most one edge from $v$ to any $V_k$ in any merge trade.  As a result, each edge in $L$ requires us to expend one edge in $T_i$ and $\frac{3}{2}$ edges in $T$ to deal with it, and therefore tells us that
\begin{align*}
\epsilon_{T_i \setminus R} \leq \epsilon_{T_i} + \frac{\delta(L)}{n}, \epsilon_{T \setminus R} \leq  \epsilon_{T} + \frac{\delta(L)}{2n}.
\end{align*}

In particular, this tells us $\epsilon$, and therefore gives us the claimed condition under which these trades can be made.

Finally, note that we've made at most $\frac{3}{4}|E(L)|$ merge trades through this entire process, each of which used at most two triangles from any $T_i$ and seven from $T$, and at most$\frac{1}{6}|E(L)|$ merge trades, each of which used at most four triangles from any $T$.  As a result, we have
\begin{align*}
\xi_{T_i \setminus R} \leq \xi_{T_i} + \frac{9|E(L)|}{2n^2},  \xi_{T \setminus R} \leq \xi_{T} + \frac{71|E(L)|}{36n^2}.
\end{align*}
\end{proof}

Apply this lemma to the graph $G$ that we have been working with throughout this paper, and let $T^{(3)} = T^{(2)} \setminus R, T_i^{(3)} = T_i^{(2)} \setminus R$: we then have

\begin{alignat*}{4}
&\epsilon_{T^{(3)}} &&\leq \epsilon_{T^{(2)}} + \frac{\delta(L)}{2n} &&\leq \epsilon_{T'} +  \frac{(\epsilon_{G'} + 4\sqrt{3\xi_{G'}} + 27\xi_{G'})9n}{2n} && \leq \frac{15}2 \epsilon_{G'} +  18\sqrt{3\xi_{G'}} + \frac{243}{2}\xi_{G'},\\
&\epsilon_{T_i^{(3)}} &&\leq \epsilon_{T_i^{(2)}} + \frac{\delta(L)}{n}  &&\leq \epsilon_{T_i'} +  \frac{(\epsilon_{G'} + 4\sqrt{3\xi_{G'}}+ 27\xi_{G'})9n}{n} && \leq 18 \epsilon_{G'} +  36\sqrt{3\xi_{G'}} + 243\xi_{G'},\\
&\xi_{T^{(3)}} &&\leq \xi_{T^{(2)}} +  \frac{71|E(L)|}{36n^2} && \leq  7\xi_{T'} + \frac{71\left(8\xi_{G'} +  \sqrt{\frac{\xi_{G'}}{3}}\right)(9n)^2}{36n^2}  && < 1278\xi_{G'} + 93\sqrt{\xi_G},\\
&\xi_{T_i^{(3)}} &&\leq \xi_{T_i^{(2)}} +  \frac{9|E(L)|}{2n^2} && \leq  7\xi_{T'} + \frac{9\left(8\xi_{G'} +  \sqrt{\frac{\xi_{G'}}{3}}\right)(9n)^2}{2n^2}  && < 2916\xi_{G'} + 632\sqrt{\xi_G}.\\
\end{alignat*}

  These subgraphs $T, T_i$ are the last sets of edges that we desire to decompose into triangles.  If we can ensure that the complement of these four graphs admit triangulations, then our results on partial Latin squares will suffice to triangulate the rest of our graph.  The following lemma is designed to do just that:

\begin{lem}\label{tricomp}
Take any tripartite graph $T$ on $3n$ vertices.  If
\begin{align*}
n - 8\epsilon_T n > 3
\end{align*}
there is a subgraphs $R$ of $H$ with the following properties:
\begin{enumerate}
\item  $R$  is made entirely out of edge-disjoint triangles.
\item  $R \cup \overline{T}$ can be decomposed entirely into edge-disjoint triangles. 
\item $\epsilon_{T \setminus R} \leq 2\epsilon_T, $.
\item $\xi_{T \setminus R} \leq 8\xi_{T}$.
\end{enumerate}
\end{lem}
\begin{proof}
Because $T$ is tridivisible, the tripartie complement $\overline{T}$ of $T$  is tridivisible.  As a consequence, any positively-oriented path (i.e. one in which we go from $V_i$ to $V_{i+1}$) can be extended to one that self-intersects, and thereby generate a positively-oriented cycle; as a consequence, we can decompose $\overline{T}$ into a collection of positively-oriented cycles whose lengths are all multiples of three.  Do this.

Consequently, we can repeatedly use the``1-2-3-1-2-3'' shrink trades in Figure \ref{22hex} to reduce the graph $\overline{T}$ into triangles, using only seven triangles from $T$ per ``1-2-3-1-2-3'' trade.  Let $R$ denote the collections of triangles needed to create these trades; by the same logic as in Lemma \ref{pathlem}, this takes at most $\frac{\xi_T n}{3}$ trades, and results in a collection $R$ such that
\begin{align*}
\epsilon_{T \setminus R} \leq 2\epsilon_T ,\qquad  \xi_{T \setminus R} \leq 8\xi_{T} \\
\end{align*}
and can be accomplished as long as
\begin{align*}
n - 8\epsilon_T n > 1.
\end{align*}
\end{proof}

Apply this lemma to each of $T^{(3)}, T_i^{(3)}$ and delete the resulting sets of trades to get graphs $T^{(4)}, T_i^{(4)}.$  These graphs have
\begin{alignat*}{3}
&\epsilon_{T^{(4)}} &&\leq 15 \epsilon_{G'} +  36\sqrt{3\xi_{G'}} + 243\xi_{G'} &&< 15 \epsilon_G + 243\xi_G + 36\sqrt{3\xi_G} + \frac{13807}{|V(G)|} + \frac{452}{\sqrt{|V(G)|}} \\
&\epsilon_{T_i^{(4)}} && \leq 36 \epsilon_{G'} +  72\sqrt{3\xi_{G'}} + 486\xi_{G'} && < 36 \epsilon_{G} +  72\sqrt{3\xi_{G}} + 486\xi_{G}  + \frac{28035}{|V(G)|} + \frac{904}{\sqrt{|V(G)|}}\\
&\xi_{T^{(4)}} && < 10224\xi_{G'} + 744\sqrt{\xi_G'} &&< 10224\xi_{G'} + 744\sqrt{\xi_G'} + \frac{536760}{|V(G)|} + \frac{5390}{\sqrt{|V(G)|}}\\
&\xi_{T_i^{(4)}} &&< 23328\xi_{G'} + 5056\sqrt{\xi_G'}&&< 23328\xi_{G'} + 5056\sqrt{\xi_G'} + \frac{1224720}{|V(G)|} + \frac{265440}{\sqrt{|V(G)|}}\\
\end{alignat*}

Therefore, for any $\epsilon < \frac{1}{432}$, we can find some constant $\xi > 0$ such that any tridivisible graph $G$ with $\epsilon_G < \epsilon, \xi_G < \xi$ has $\epsilon_{T^{(4)}}, \epsilon_{T_i^{(4)}} < \frac{1}{12}$.  Because all graphs where $|V(G)| > \frac{1}{\sqrt{\xi_G}}$ are complete, Kirkman's result takes care of these cases, and leaves us with just those for which we can apply the crude bound  $|V(G)| \leq \frac{1}{\sqrt{\xi_G}}$.  In particular, this tells us that a sufficiently small choice of $\xi$ will yield a graph $G$ which after these trades will yield tripartite graphs that Theorem \ref{bartthm} applies to, thus giving us a triangulation of $G$.  In short, we have proven the main claim of this paper:

\begin{thm}
For any $\epsilon < \frac{1}{432}$, there is some $\xi > 0$ such that any tridivisible graph $G$ with $\epsilon_G < \epsilon, \xi_G < \xi$ admits a triangle decomposition.
\end{thm}

\normalsize
\bibliographystyle{plain}	
\bibliography{myrefs}

\end{document}